\documentclass{article}
\usepackage[margin=1in]{geometry}
\usepackage[utf8]{inputenc}
\usepackage{mathrsfs}
\usepackage{amsmath, amsfonts, amsthm, bm, varwidth, hyperref, xfrac, textgreek, mathtools, booktabs, tabularx, array, float, xcolor, mdframed, amssymb}
\usepackage{caption}
\usepackage{algorithm}
\usepackage[noend]{algpseudocode}
\usepackage{enumitem}
\usepackage{multirow}
\usepackage{siunitx}
\usepackage{scalerel}
\usepackage{mleftright}
\mleftright
\usepackage{tikz}
\usetikzlibrary{backgrounds,matrix,fit,decorations.pathreplacing,calc,positioning,calligraphy}
\usepackage{arydshln}
\newcommand{\arrowlabelsep}{4pt}

\DeclareMathOperator{\im}{im}
\DeclareMathOperator{\rank}{rank}

\usepackage[backend=biber,style=alphabetic,maxalphanames=5,maxnames=8]{biblatex}
\newcolumntype{C}{>{\raggedleft\arraybackslash}X} 
\newtheorem{theorem}{Theorem}[section]
\newtheorem{lemma}[theorem]{Lemma}
\newtheorem{proposition}[theorem]{Proposition}

\theoremstyle{definition}
\newtheorem{definition}[theorem]{Definition}
\newtheorem{remark}[theorem]{Remark}
\newtheorem{example}[theorem]{Example}
\newtheorem*{theorem*}{Theorem}
\newtheorem*{conjecture*}{Conjecture}

\DeclarePairedDelimiterX{\set}[1]{\{}{\}}{\setargs{#1}}
\NewDocumentCommand{\setargs}{>{\SplitArgument{1}{;}}m}
{\setargsaux#1}
\NewDocumentCommand{\setargsaux}{mm}
{\IfNoValueTF{#2}{#1} {#1\nonscript\:\delimsize\vert\allowbreak\nonscript\:\mathopen{}#2}}%
\allowdisplaybreaks

\algrenewcommand\textproc{}
\title{On the Integral Cohomology of Fano Varieties of Linear Subspaces} 
\author{Benjamin E. \textsc{Diamond} \\
	\scriptsize \href{mailto:benediamond@gmail.com}{\texttt{benediamond@gmail.com}}
}
\date{}
\newcommand{\underscore}[0]{{\rule[0ex]{0.5em}{0.6pt}}}

\renewcommand{\setminus}{-}

\begin{document}
\maketitle

\begin{abstract}
For each $n$, each dimension $r$, and each subscheme $X \subset \mathbb{P}^n$ defined as the common zero-locus of $s$ hypersurfaces, of degrees $\mathbf{d} = (d_1, \ldots , d_s)$ say,
the \textit{Fano scheme} $F_r(X)$ of projective $r$-spaces contained in $X$ is a subscheme of the Grassmannian $G(r + 1, n + 1)$. We prove that the inclusion $F_r(X) \subset G(r + 1, n + 1)$ induces an isomorphism $H^i(G(r + 1, n + 1); \mathbb{Z}) \rightarrow H^i(F_r(X); \mathbb{Z})$ on \textit{integral} cohomology for certain indices $i$ (i.e., depending only on $n$, $r$, $s$ and $\mathbf{d}$). Our result extends to the integral setting a result proved for rational cohomology by Debarre and Manivel (Math.\ Ann.\ '98), and answers a question of Benoist and Voisin. Our techniques adapt ones introduced by Tu (Trans.\ Am.\ Math.\ Soc.\ '89) for a different purpose.

\end{abstract}

\section{Introduction}

In a now-classic work, Debarre and Manivel \cite[Thm.~3.4]{Debarre:1998aa} prove---among many other things---that, over the complex numbers, for each dimension $n$, each subspace dimension $r$, each multidegree $\mathbf{d} = (d_1, \ldots , d_s)$, and each complete intersection $X \subset \mathbb{P}^n$ of multidegree $\mathbf{d}$ for which the \textit{Fano scheme} $F_r(X)$ of projective $r$-spaces contained in $X$ is of the expected dimension and smooth, the restriction map $H^i(G(r + 1, n + 1), \mathbb{Q}) \rightarrow H^i(F_r(X), \mathbb{Q})$ on degree-$i$ rational cohomology induced by the inclusion $F_r(X) \subset G(r + 1, n + 1)$ is an isomorphism whenever $i \in \{0, \ldots , \delta\underscore - 1\}$ is small enough (here, $\delta\underscore := \delta\underscore(n, \mathbf{d}, r)$ is a constant that depends only on $n$, $r$, $s$ and $\mathbf{d}$). 

In this work, we prove that \cite[Thm.~3.4]{Debarre:1998aa} holds over $\mathbb{Z}$, too. In fact, we relax the setting of \cite[Thm.~3.4]{Debarre:1998aa} in various further ways. We don't assume that the forms $\mathbf{f} = (f_1, \ldots , f_s)$ are general, or that their common zero-locus $X \subset \mathbb{P}^n$ is smooth or of the expected dimension. Further, we refrain from demanding that $F_r(X)$ itself be smooth or of the expected dimension (it could be singular, or even nonreduced).

Our techniques are quite different from those used by Debarre and Manivel; instead, they adapt ones used by Tu \cite{Tu:1989aa} in a rather different context. Tu relates his primary object of study---the rank-$2 \cdot (n - r)$ degeneracy locus of a symmetric bundle map $V \otimes V  \to \mathbb{C}$ over a variety, where $V$ is of rank $n + 1$, say---to the vanishing locus in $G$, the Grassmann bundle of $r + 1$-dimensional subspaces of fibers of $V$, of a certain section of $\text{Sym}^2 S^*$, where $S$ is the tautological sub-bundle over $G$. Tu establishes the connectedness of the degeneracy locus at issue by linking it to the vanishing of the top two cohomology groups of the \textit{complement} in $G$ of that latter vanishing locus.

If $S^*$, and by extension its symmetric powers, were ample over $G(r + 1, n + 1)$, then Tu's result---and ours---would be essentially immediate (see e.g.\ Lazarsfeld \cite[Thm.~7.1.1]{Lazarsfeld:2004ab}), a fact Tu already notes. Tu's technique yields a kind of workaround for this problem; indeed, Tu instead establishes the required cohomology vanishing using a key cohomology comparison lemma \cite[Lem.~3.6]{Tu:1989aa} based on the Leray spectral sequence, applying that lemma to a certain induced map $h : \mathbb{P}(\text{Sym}^2 S) \to \mathbb{P}(\text{Sym}^2 V)$. To apply his lemma, Tu must study $h$ in detail, and in particular must stratify the target space $\mathbb{P}(\text{Sym}^2 V)$ by fiber dimension, analyzing both the dimensions of the strata and the dimensions of their associated fibers.

Tu's apparatus perfectly suits our problem, though it originates as an auxiliary tool in the service of his. That is, we have no base variety or degeneracy locus in Tu's sense. Tu's vanishing locus in $G$ is merely a means to his end; in our setting, it's nothing other than our object of immediate interest, the Fano scheme $F_r(X) \subset G(r + 1, n + 1)$, which arises in $G$ as the vanishing locus of a section of $\text{Sym}^\mathbf{d} S^*$. We must simply replace Tu's map with its higher-degree analogue $h : \mathbb{P}(\text{Sym}^\mathbf{d} S) \to \mathbb{P}(\text{Sym}^\mathbf{d} V)$; moreover, we must assert the vanishing of the top $\delta\underscore + 1$---as opposed to the top two---cohomology groups of the appropriate open subset of $G$ (we describe our adaptation of Tu's techniques further in Section \ref{tu's} below).

Our map $h : \mathbb{P}(\text{Sym}^\mathbf{d} S) \to \mathbb{P}(\text{Sym}^\mathbf{d} V)$ presents new challenges. In the classical case $\mathbf{d} = (2)$, elements $[\varphi] \in \mathbb{P}(\text{Sym}^2 V)$ are essentially symmetric bilinear forms; the dimension of each fiber $h^{-1}\left( \{[\varphi]\}\right)$ depends precisely on $\text{rank}(\varphi)$, and the loci in $\mathbb{P}(\text{Sym}^2 V)$ consisting of those $[\varphi]$ for which $\rank (\varphi)$ is prescribed are easy to understand. In the setting of general $\mathbf{d}$, we must find a suitable generalization of the notion ``rank'', and again analyze the ensuing strata and their associated fiber dimensions. To do this, we make use of \textit{apolarity}, a classical notion reprised in various modern treatments (see e.g.\ Iarrobino and Kanev \cite[Def.~1.11]{Iarrobino:1999aa}). For each $\bm{\varphi} \in \text{Sym}^\mathbf{d} V$, we write $A(\bm{\varphi}) \subset V^*$ for the set of forms $\ell \in V^*$ \textit{apolar} to $\bm{\varphi}$, in the sense that they annihilate $\bm{\varphi}$ under the \textit{contraction action} $\ell : \bm{\varphi} \mapsto \ell \left( \frac{\partial}{\partial x_0}, \ldots , \frac{\partial}{\partial x_n} \right) (\bm{\varphi})$ (this is essentially a directional derivative). For $\bm{\varphi}$ of arbitrary multidegree $\mathbf{d}$, our $A(\bm{\varphi}) \subset V^*$ plays precisely the rôle played in the special case $\mathbf{d} = (2)$ by $\text{ker}\left( \varphi : V^* \to V \right)$, and our $\text{dim}(A(\bm{\varphi})^\perp)$ exactly generalizes $\text{rank}(\varphi)$ (we discuss this fact further in Example \ref{tu_specialize}). Our approach, in a sense, reveals new power in Tu's original technique, since we establish the vanishing of $\delta\underscore + 1$ cohomology groups (as opposed to just $2$).


Our work answers a question of O.\ Benoist and Voisin, 
who have asked whether this work's result holds in the special case $s = 1$, $r = 1$, $\mathbf{d} = (n)$, and $\mathbf{f} = (f_1)$ general. In this setting, $X \subset \mathbb{P}^n$ is a smooth, degree-$n$ hypersurface, and $F_1(X) \subset G(1 + 1, n + 1)$, the Fano variety of lines contained in $X$, is smooth of dimension $\delta\underscore = n - 3$; moreover, $\delta\underscore = \delta = n - 3$ in this case, so that one expects isomorphisms $H^i(G(1 + 1, n + 1); \mathbb{Z}) \to H^i(F_1(X); \mathbb{Z})$ for each $i \in \{0, \ldots , \delta - 1\}$ below the middle degree of $F_1(X)$. Our work also represents a kind of progress towards Debarre and Manivel's question \cite[Rem.~3.2.~3)]{Debarre:1998aa}, which predicts the relative \textit{homotopy} vanishing $\pi_i\left( G(r + 1, n + 1), F_r(X) \right) = 0$ for each $i \in \{0, \ldots , \delta\underscore \}$.

Upon being specialized to the case $r = 0$, our work yields a new proof of the weak Lefschetz theorem for arbitrary intersections of ample divisors, itself a consequence of \textit{Sommese's theorem} (see Remark \ref{new} below). 

\paragraph{Our result.}

We work over $\mathbb{C}$ throughout. Our goal is to prove the following theorem. As in \cite{Debarre:1998aa}, we fix a dimension $n$, a subspace dimension $r$, and a multidegree $\mathbf{d} = (d_1, \ldots , d_s)$. Throughout, we let $V$ be an $n + 1$-dimensional $\mathbb{C}$-vector space. Again as in \cite{Debarre:1998aa}, we use the notational device ${\mathbf{d} + r \choose r} \coloneqq \sum_{i = 1}^s {d_i + r \choose r}$, and write $\text{Sym}^\mathbf{d} V^* \coloneqq \bigoplus_{i = 1}^s \text{Sym}^{d_i} V^*$. We write $\delta \coloneqq (r + 1) \cdot (n - r) - {\mathbf{d} + r \choose r}$ for the expected dimension of $F_r(X) \subset G(r + 1, n + 1)$. Finally, we write $\delta\underscore \coloneqq \min(\delta, n - 2r - s)$, exactly as in \cite{Debarre:1998aa}.

\begin{theorem} \label{main}
For each $n$, $r$, and $s$, each multidegree $\mathbf{d} = (d_1, \ldots , d_s)$ and each $\mathbf{f} = (f_1, \ldots , f_s)$ in $\text{Sym}^\mathbf{d} V^*$, with common zero scheme $X \subset \mathbb{P}^n$ say, the restriction map $H^i(G(r + 1, n + 1); \mathbb{Z}) \to H^i(F_r(X); \mathbb{Z})$ induced by the inclusion $F_r(X) \subset G(r + 1, n + 1)$ is an isomorphism for $i \in \{0, \ldots , \delta\underscore - 1\}$ and injective when $i = \delta\underscore$.
\end{theorem}
We don't assume that $X$ is general or that $F_r(X)$ is smooth or of the expected dimension.

\section{Tu's Approach} \label{tu's}

Here, we reproduce in our context the apparatus developed by Tu \cite{Tu:1989aa}. The situation is almost identical. We just specialize Tu's base variety (which he calls $X$) to $\text{Spec} (\mathbb{C})$, and replace $\text{Sym}^2$ with $\text{Sym}^\mathbf{d}$ everywhere.

We begin by fixing a \textit{nonzero} $s$-tuple $\mathbf{f} = (f_1, \ldots , f_s) \in \text{Sym}^\mathbf{d} V^*$ of homogeneous polynomials, just as in \cite[\S~1]{Debarre:1998aa}. (If $\mathbf{f} = 0$, then our result, though trivial, is still true; in this case, the inclusion $F_r(X) \hookrightarrow G(r + 1, n + 1)$ is the identity, and the restriction $H^i(G(r + 1, n + 1); \mathbb{Z}) \to H^i(F_r(X); \mathbb{Z})$ is an isomorphism for \textit{each} $i \in \{0, \ldots , 2 \cdot \dim G(r + 1, n + 1)\}$.)

In \cite[\S~2]{Tu:1989aa}, Tu explains how to associate, to each section $s$ of a rank-$e$ vector bundle $E$, a further section $s^*$ of the line bundle $\mathscr{O}(1)$ over $\mathbb{P}(E^*)$. 
Exactly as Tu does, we apply this construction twice. First, we view $\mathbf{f}$ as a section of $\text{Sym}^\mathbf{d} V^*$, the trivial vector bundle over $\text{Spec}(\mathbb{C})$ (our base variety is trivial, while Tu's isn't). Applying \cite[\S~2]{Tu:1989aa}, we obtain a corresponding section $\mathbf{f}^*$ of $\mathscr{O}(1)$ over $\mathbb{P}\left( \left(\text{Sym}^\mathbf{d} V^*\right)^* \right) = \mathbb{P}\left( \text{Sym}^\mathbf{d} V \right)$.

Moreover, just as in both \cite[Rem.~2.10]{Debarre:1998aa} and \cite[\S~3]{Tu:1989aa}, $\mathbf{f}$ lifts to a section $\pi^*(\mathbf{f})$ of $\pi^*\left( \text{Sym}^\mathbf{d}V^* \right)$, which in turn induces, by restriction, a further section, say $\bm{\mathfrak{f}}$, of $\text{Sym}^\mathbf{d}(S^*)$. Here, $S$ is the rank-$r + 1$ tautological sub-bundle over $G(r + 1, n + 1)$. The Fano scheme $F_r(X) \subset G(r + 1, n + 1)$ is exactly the vanishing locus $Z(\bm{\mathfrak{f}}) \subset G(r + 1, n + 1)$. Applying \cite[\S~2]{Tu:1989aa} again, we obtain from $\bm{\mathfrak{f}}$ a further section, say $\bm{\mathfrak{f}}^*$, of the line bundle $\mathscr{O}(1)$ over the projectivization $\mathbb{P}\left( \left( \text{Sym}^\mathbf{d} S^* \right)^* \right) = \mathbb{P}\left( \text{Sym}^\mathbf{d} S \right)$.

We summarize this state of affairs in the figure above, which essentially reproduces Tu's \cite[(3.4)]{Tu:1989aa} verbatim. As Tu does, we use doubled arrows to denote sections.

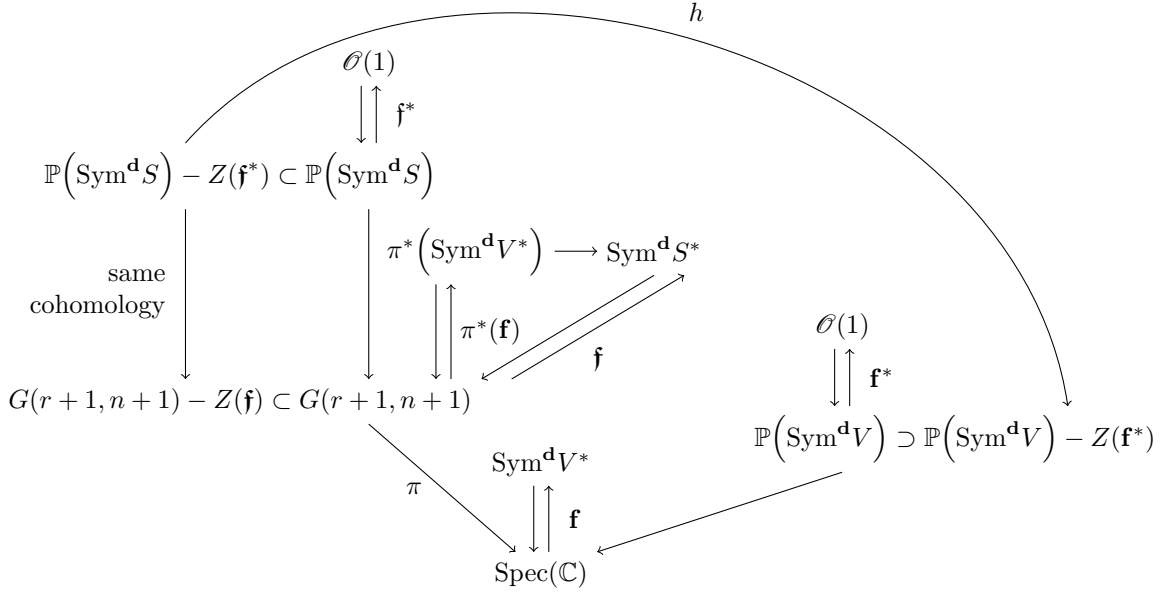
\begin{figure}
\centering
\begin{tikzpicture} 

  \node (Ptop) at (0, 4) {$\mathbb{P}\left( \text{Sym}^\mathbf{d} S \right) \setminus Z(\bm{\mathfrak{f}}^*) \subset \mathbb{P}\left( \text{Sym}^\mathbf{d} S \right)$};

  \node (G) at (0, 1) {$G(r + 1, n + 1) \setminus Z(\bm{\mathfrak{f}}) \subset G(r + 1, n + 1)$};
  \draw[->] ($(Ptop.south west) + (2, 0)$) -- node[left=\arrowlabelsep,align=right] {same\\cohomology} ($(G.north west -| Ptop.south west) + (2, 0)$);
  \draw[->] ($(G.north east |- Ptop.south west) + (-1.5, 0)$) -- ($(G.north east) + (-1.5, 0)$);
  
  \node (new) at ($(G.north east |- Ptop.west) + (-1.5, 1.5)$) {$\mathscr{O}(1)$};  
  
  \draw[->] ($(new.south) + (-0.1, 0)$) -- ($(Ptop.north -| new.south) + (-0.1, 0)$);
  \draw[->] ($(Ptop.north -| new.south) + (0.1, 0)$) -- node[right=\arrowlabelsep] {$\bm{\mathfrak{f}}^*$} ($(new.south) + (0.1, 0)$);

  \node (piSymE) at (3, 3) {$\pi^* \left( \mathrm{Sym}^\mathbf{d} V^* \right)$};
  \node (SymS)   at (5.5, 3) {$\mathrm{Sym}^\mathbf{d} S^*$};
  \draw[->] (piSymE) -- (SymS);
  \draw[->] ($(piSymE.south) + (-0.4, 0)$) -- ($(G.north east -| piSymE.south) + (-0.4, 0)$);
  \draw[->] ($(G.north east -| piSymE.south) + (-0.2, 0)$) -- node[right] {$\pi^*(\mathbf{f})$} ($(piSymE.south) + (-0.2, 0)$);
  \draw[->] (SymS.south) -- (G.north east);
  \draw[->] ($(G.north east) + (0.4, 0)$) -- node[below=\arrowlabelsep] {$\bm{\mathfrak{f}}$} ($(SymS.south) + (0.4, 0)$);

  \node (X)    at (4, -1.3) {$\text{Spec}\left( \mathbb{C} \right)$};
  \node (SymE) at (4, 0.2)  {$\mathrm{Sym}^\mathbf{d} V^*$};
  \draw[->] ($(G.south east) + (-1.5, 0)$) -- node[left=\arrowlabelsep] {$\pi$} (X);
  \draw[->] ($(SymE.south) + (-0.1, 0)$) -- ($(X.north) + (-0.1, 0)$);
  \draw[->] ($(X.north) + (0.1, 0)$) -- node[right=\arrowlabelsep] {$\mathbf{f}$} ($(SymE.south) + (0.1, 0)$);

  \node (PSymE) at (9.5, 0.5) {$\mathbb{P}\left( \mathrm{Sym}^\mathbf{d} V \right) \supset \mathbb{P}\left( \mathrm{Sym}^\mathbf{d} V \right) \setminus Z(\mathbf{f}^*)$};
  \draw[->] ($(PSymE.south) + (-1.5, 0)$) -- (X.north east);

  \node (O1) at (8, 2) {$\mathscr{O}(1)$};
  \draw[->] ($(O1.south) + (-0.1, 0)$) -- ($(PSymE.north -| O1.south) + (-0.1, 0)$);
  \draw[->] ($(PSymE.north -| O1.south) + (0.1, 0)$) -- node[right=\arrowlabelsep] {$\mathbf{f}^*$} ($(O1.south) + (0.1, 0)$);

\draw[->]
  ($(Ptop.north west) + (2,0)$)
  to[bend left=65]
  node[above=\arrowlabelsep] {$h$}
  ($(PSymE.north) + (1.5,0)$);
\end{tikzpicture}
\end{figure}

We're interested in the cohomology of $Z(\bm{\mathfrak{f}}) = F_r(X)$. On the other hand, we have the following topological statement, which adapts \cite[Lem.~3.3]{Tu:1989aa}. We abbreviate $G \coloneqq G(r + 1, n + 1)$ and $F \coloneqq F_r(X)$. 
\begin{lemma} \label{topological}
If $H^i\left( G \setminus F; \mathbb{Z}\right) = 0$ holds for each $i \in \{2 \cdot \dim G - \delta\underscore, \ldots , 2 \cdot \dim G\}$, then Theorem \ref{main} holds.
\end{lemma}
\begin{proof}
We proceed as \cite[Lem.~3.3]{Tu:1989aa} does. By Lefschetz duality, for each $i \in \{0, \ldots , 2 \cdot \dim G\}$, we have
\begin{equation*}H^{2 \cdot \dim G - i}(G \setminus F; \mathbb{Z}) \cong H_i(G, F; \mathbb{Z});\end{equation*}
in particular, if the hypothesis of the lemma holds, then, for each $i \in \{0, \ldots , \delta\underscore\}$, $H_i(G, F; \mathbb{Z}) = 0$ also does. By the universal coefficient theorem for relative cohomology, we have, for each $i \in \{0, \ldots , 2 \cdot \dim G\}$, the exact sequence:
\begin{equation*}
0 \to \text{Ext}(H_{i - 1}(G, F), \mathbb{Z}) \to H^i(G, F; \mathbb{Z}) \to \text{Hom}(H_i(G, F), \mathbb{Z}) \to 0.
\end{equation*}
We conclude under the hypothesis of the lemma the relative cohomology vanishing $H^i(G, F; \mathbb{Z}) = 0$ for each $i \in \{0, \ldots , \delta\underscore\}$. Finally, using the long exact sequence of relative cohomology
\begin{equation*}
0 \to H^0(G, F) \to H^0(G) \to H^0(F) \to \cdots \to H^i(G, F) \to H^i(G) \to H^i(F) \to H^{i + 1}(G, F) \to \cdots,
\end{equation*}
we obtain, again under the hypothesis of the lemma, isomorphisms $H^i(G; \mathbb{Z}) \xrightarrow{\sim} H^i(F; \mathbb{Z})$ for each $i \in \{0, \ldots , \delta\underscore - 1\}$ and an injection $H^i(G; \mathbb{Z}) \hookrightarrow H^i(F; \mathbb{Z})$ for $i = \delta\underscore$; this is exactly the statement of Theorem \ref{main}.
\end{proof}

Moreover, as Tu explains \cite[Cor.~2.2]{Tu:1989aa} (Tu elsewhere \cite[Prop.~3.1]{Tu:1990aa} attributes this result to Sommese \cite[Prop.~1.16]{Sommese:1978aa}), $\mathbb{P}\left( \text{Sym}^\mathbf{d} S \right) \setminus Z(\bm{\mathfrak{f}}^*)$ and $G(r + 1, n + 1) \setminus Z(\bm{\mathfrak{f}})$ have the same cohomology.
\begin{lemma} \label{same}
For each $i \in \{0, \ldots , 2 \cdot \dim G\}$, $H^i\left( \mathbb{P}\left( \mathrm{Sym}^\mathbf{d} S \right) \setminus Z(\bm{\mathfrak{f}}^*); \mathbb{Z} \right) \cong H^i \left( G(r + 1, n + 1) \setminus Z(\bm{\mathfrak{f}}); \mathbb{Z} \right)$.
\end{lemma}
\begin{proof}
This is simply \cite[Cor.~2.2]{Tu:1989aa}.
\end{proof}

In light of Lemmas \ref{topological} and \ref{same}, it suffices to prove that $H^i\left( \mathbb{P}\left( \mathrm{Sym}^\mathbf{d} S \right) \setminus Z(\bm{\mathfrak{f}}^*); \mathbb{Z} \right) = 0$ for each $i \in \{2 \cdot \dim G - \delta\underscore , \ldots , 2 \cdot \dim G\}$.

Once again as in \cite{Tu:1989aa}, we note that $\mathscr{O}(1)$ is ample on $\mathbb{P}\left( \mathrm{Sym}^\mathbf{d} V \right)$, so that $\mathbb{P}\left( \text{Sym}^\mathbf{d} V \right) \setminus Z(\mathbf{f}^*)$ is affine. Tu moreover proves the following cohomology comparison lemma, which we reproduce almost verbatim. 
We write $\mathbb{N} \coloneqq \{0, 1, 2, \ldots \}$.

\begin{lemma}[{Tu \cite[Lem.~3.6]{Tu:1989aa}}] \label{comparison}
Let $h : M \to Y$ be a surjective proper morphism from any variety $M$ to an affine variety $Y$. Suppose there is a strictly increasing function $d : \mathbb{N} \to \mathbb{N}$ and a sequence of closed algebraic subsets
\begin{equation*} \cdots \subset Y_{k + 1} \subset Y_k \subset \cdots \subset Y_0 = Y\end{equation*}
such that for each $x \in Y_k \setminus Y_{k + 1}$,
\begin{equation*}d(k) = \dim h^{-1}\left( x \right).\end{equation*}
Define
\begin{equation*}R \coloneqq \max_{k \geq 0} \left( \dim Y_k + 2 \cdot d(k) \right).\end{equation*}
Then for each $i > R$, $H^i(M; \mathbb{Z}) = 0$.
\end{lemma}
\begin{remark}
Though Tu \cite[Lem.~3.6]{Tu:1989aa} writes ``closed subvarieties'' where we write ``closed algebraic subsets'', his proof goes through unchanged even on reducible (though closed) strata. In any case, when we apply Lemma \ref{comparison}, our strata \textit{will} in fact be irreducible, so the point is actually immaterial (see Remark \ref{hypothesis}).
\end{remark}

We note that the analogue of \cite[Prop.~3.5]{Tu:1989aa} holds in our setting; that is, the map $h : \mathbb{P}\left( \text{Sym}^\mathbf{d} S \right) \to \mathbb{P}\left( \text{Sym}^\mathbf{d} V \right)$ sends $\mathbb{P}\left( \text{Sym}^\mathbf{d} S \right) \setminus Z(\bm{\mathfrak{f}}^*)$ into $\mathbb{P}\left( \mathrm{Sym}^\mathbf{d} V \right) \setminus Z(\mathbf{f}^*)$. Moreover, using the equality $\bm{\mathfrak{f}}^* = h^*(\mathbf{f}^*)$, one can show that the restriction of $h$ to $\mathbb{P}\left( \text{Sym}^\mathbf{d} S \right) \setminus Z(\bm{\mathfrak{f}}^*)$ is precisely the base-change of $h$ over $\mathbb{P}\left( \mathrm{Sym}^\mathbf{d} V \right) \setminus Z(\mathbf{f}^*)$, and so is proper. Our goal is to apply Lemma \ref{comparison} to our map $h : \mathbb{P}\left( \text{Sym}^\mathbf{d} S \right) \setminus Z(\bm{\mathfrak{f}}^*) \to \mathbb{P}\left( \mathrm{Sym}^\mathbf{d} V \right) \setminus Z(\mathbf{f}^*)$ (viewed as a surjective proper map onto its image). We recall \cite{Tu:1989aa} that $h$ sends the pair $(\Lambda, [\bm{\varphi}])$ to $([\bm{\varphi}])$, where $\Lambda \subset V$ is $r + 1$-dimensional and $\bm{\varphi} \in \text{Sym}^\mathbf{d} \Lambda$; i.e., $h$ 
is the ``forgetful map'' induced by the inclusion $\Lambda \subset V$.

\section{Main Result}

We now initiate a detailed study of the map $h : \mathbb{P}\left( \text{Sym}^\mathbf{d} S \right) \to \mathbb{P}\left( \mathrm{Sym}^\mathbf{d} V \right)$. Our first technical goal is to develop a criterion that predicts the dimension of $h^{-1}(\{[\bm{\varphi}]\})$, for $\bm{\varphi}$ an arbitrary nonzero element of $\mathrm{Sym}^\mathbf{d} V$.

For each nonzero $\bm{\varphi} \in \mathrm{Sym}^\mathbf{d} V$, the fiber $h^{-1}(\{[\bm{\varphi}]\})$ is the set of pairs $(\Lambda, [\bm{\varphi}])$ for which $\bm{\varphi} \in \text{Sym}^\mathbf{d} \Lambda$. To study $h^{-1}(\{[\bm{\varphi}]\})$, it's thus enough to figure out which subspaces $\Lambda \subset V$ satisfy $\bm{\varphi} \in \text{Sym}^\mathbf{d} \Lambda$. Our goal is to reduce the question $\bm{\varphi} \stackrel{?}\in \text{Sym}^\mathbf{d} \Lambda$ to one of simple linear-algebraic containment, of the flavor $M(\bm{\varphi}) \stackrel{?}\subset \Lambda$. To this end, we prove the following lemma. 

\begin{lemma} \label{universal}
For each $\bm{\varphi} \in \mathrm{Sym}^\mathbf{d} V$, there exists a unique linear subspace $M(\bm{\varphi}) \subset V$ fulfilling the universal property whereby for each linear subspace $W \subset V$, $\bm{\varphi} \in \mathrm{Sym}^\mathbf{d} W$ if and only if $M(\bm{\varphi}) \subset W$. 
\end{lemma}
\begin{proof}
If $M(\bm{\varphi}) \subset V$ exists, then its uniqueness is immediate. We turn to $M(\bm{\varphi})$'s existence.

There is a well-defined \textit{contraction action} between elements $\ell \in V^*$ and elements $\bm{\varphi} \in \mathrm{Sym}^\mathbf{d} V$ (this action appears in e.g.\ Iarrobino and Kanev \cite[(1.1.2)]{Iarrobino:1999aa} and Eisenbud \cite[Prop.~A2.8]{Eisenbud:1995aa}). In concrete terms, it's enough to pick a basis $x_0, \ldots , x_n$ of $V$, with dual basis $X_0, \ldots , X_n$ say, to identify $\bm{\varphi} = (\varphi_1, \ldots ,\varphi_s) \in \mathrm{Sym}^\mathbf{d} V$ with a multidegree-$\mathbf{d}$ tuple of homogeneous polynomials in the indeterminates $x_0, \ldots , x_n$, to write $\ell = \ell_0 \cdot X_0 + \cdots + \ell_n \cdot X_n$ in the dual basis, and finally to write
\begin{equation*}
\renewcommand{\arraystretch}{2.3}
\setlength{\extrarowheight}{-2pt}
\setlength{\arraycolsep}{7pt}
\ell \circ \bm{\varphi} \coloneqq \begin{bmatrix}
\begin{array}{c c c}
\ell_0 & \cdots & \ell_n
\end{array}
\end{bmatrix} \cdot
\begin{bmatrix}
\begin{array}{c c c}
\frac{\partial \varphi_1}{\partial x_0} & \cdots & \frac{\partial \varphi_s}{\partial x_0} \\
\vdots & \ddots & \vdots \\
\frac{\partial \varphi_1}{\partial x_n} & \cdots & \frac{\partial \varphi_s}{\partial x_n}
\end{array}
\end{bmatrix}.
\end{equation*}
The resulting object $\ell \circ \bm{\varphi}$ is an element of $\text{Sym}^{\mathbf{d} - 1} V$.

We define
\begin{equation*}A(\bm{\varphi}) \coloneqq \set*{\ell \in V^* ; \ell \circ \bm{\varphi} = 0}\end{equation*}
and
\begin{equation*}M(\bm{\varphi}) \coloneqq A(\bm{\varphi})^\perp.\end{equation*}
Our goal is to argue that $M(\bm{\varphi}) \subset V$ fulfills the universal property in the statement of the lemma.

We first handle the case $s = 1$. We let $\mathbf{d} = (d)$ say, and $\bm{\varphi} = (\varphi)$. The \textit{if} direction of the lemma's universal property is equivalent to the claim whereby $\varphi \in \text{Sym}^d M(\varphi)$. Up to changing our basis, we can assume that $A(\varphi) = \left< X_0, \ldots , X_{\dim A(\varphi) - 1} \right>$, so that moreover $M(\varphi) = \left< x_{\dim A(\varphi)}, \ldots , x_n \right>$. For each $i \in \{0, \ldots , \dim A(\varphi) - 1\}$, since $X_i \in A(\varphi)$, we have $\frac{\partial \varphi}{\partial x_i} = 0$; this fact implies that $\varphi$ is independent of the indeterminate $x_i$. We see that $\varphi$ depends only on the indeterminates $x_{\dim A(\varphi)}, \ldots , x_n$; in other words, $\varphi \in \text{Sym}^d M(\varphi)$. Conversely, we fix a subspace $W \subset V$ for which $\varphi \in \text{Sym}^d W$. We again pick a basis $x_0, \ldots , x_n$ of $V$, now in such a way that $W = \left< x_0, \ldots , x_{\dim W - 1} \right>$. Our hypothesis $\varphi \in \text{Sym}^d W$ implies that $\varphi$ is expressible as a degree-$d$ homogeneous polynomial in the indeterminates $x_0, \ldots , x_{\dim W - 1}$. For each $\ell \in W^\perp$, $\ell \circ \varphi$ is a linear combination of the partial derivatives $\frac{\partial \varphi}{\partial x_{\dim W}}, \ldots , \frac{\partial \varphi}{\partial x_n}$, each of which individually vanishes, by hypothesis on $\varphi$. We have just shown that $W^\perp \subset A(\varphi)$, or equivalently $M(\varphi) \subset W$. This completes the proof of the case $s = 1$.

In the case of general $s$, we note that $A(\bm{\varphi}) = \bigcap_{i = 1}^s A(\varphi_i)$, so that $A(\bm{\varphi}) \subset A(\varphi_i)$ holds for each $i \in \{1, \ldots , s\}$, and in particular $M(\varphi_i) = A(\varphi_i)^\perp \subset A(\bm{\varphi})^\perp = M(\bm{\varphi})$ also does. Using the \textit{if} direction of the universal property individually for each $i \in \{1, \ldots , s\}$, we find that $\varphi_i \in \text{Sym}^{d_i} M(\varphi_i)$ for each $i \in \{1, \ldots , s\}$. By the containment just above, we conclude in turn that $\varphi_i \in \text{Sym}^{d_i} M(\bm{\varphi})$ for each $i \in \{1, \ldots , s\}$, or in other words that $\bm{\varphi} \in \text{Sym}^\mathbf{d} M(\bm{\varphi})$. This proves the \textit{if} direction. Conversely, for each $W \subset V$ for which $\varphi_i \in \text{Sym}^{d_i} W$ holds for each $i \in \{1, \ldots , s\}$, applying the \textit{only if} direction of the universal property individually, we find that $W \supset M(\varphi_i) = A(\varphi_i)^\perp$ for each $i \in \{1, \ldots , s\}$; in particular, $W \supset A(\varphi_1)^\perp + \cdots + A(\varphi_s)^\perp = \left( A(\varphi_1) \cap \cdots \cap A(\varphi_s) \right)^\perp = A(\bm{\varphi})^\perp = M(\bm{\varphi})$, which proves the \textit{only if} direction.
\end{proof}

\begin{definition}
For each $\bm{\varphi} \in \mathrm{Sym}^\mathbf{d} V$, we write $m(\bm{\varphi}) \coloneqq \dim M(\bm{\varphi})$.
\end{definition}

\begin{remark}
We could have alternatively defined the subspace $M(\bm{\varphi}) \subset V$ of Lemma \ref{universal} as $M'(\bm{\varphi}) \coloneqq \bigcap_{\bm{\varphi} \in \text{Sym}^\mathbf{d} W} W$, i.e.\ as the intersection in $V$ of all those subspaces $W$ for which $\bm{\varphi} \in \text{Sym}^\mathbf{d} W$. This subspace $M'(\bm{\varphi})$ clearly satisfies the \textit{only if} direction of the universal property, but is $\bm{\varphi} \in \text{Sym}^\mathbf{d} M'(\bm{\varphi})$? This inclusion indeed holds, but proving that it does is not free, and seems to demand something like the proof of Lemma \ref{universal}. With the conclusion of Lemma \ref{universal} in hand, of course, the conclusion $M'(\bm{\varphi}) = M(\bm{\varphi})$ becomes immediate. 
\end{remark}

\begin{example} \label{tu_specialize}
Tu's main result \cite{Tu:1989aa} leverages the fact that elements $\varphi \in \text{Sym}^2 V$ correspond to maps $\varphi : V^* \to V$; moreover, under this correspondence, a subspace $W \subset V$ satisfies $\varphi \in \text{Sym}^2 W$ if and only if $\im (\varphi) \subset W$. We claim that in the case $\mathbf{d} = (2)$, our $M(\varphi) \subset V$ exactly specializes to $\im (\varphi) \subset V$. Indeed, each $\varphi \in \text{Sym}^2 V$ identifies to the symmetric bilinear form $B_\varphi : V^* \times V^* \to \mathbb{C}$ defined by $B_\varphi(\ell, \ell') \coloneqq (\ell \otimes \ell')(\varphi)$, and hence to the quadratic form $f(\ell') = \frac{1}{2} \cdot B_\varphi(\ell', \ell')$ on $V^*$. For $\ell \in V^*$ arbitrary, $\ell \circ \varphi$ is the linear form
\begin{equation*}\ell' \mapsto (\ell(\partial) f)(\ell') = \left. \frac{\partial}{\partial t} \right|_{t = 0} f(\ell' + t \cdot \ell) = \left. \frac{\partial}{\partial t} \right|_{t = 0} \frac{1}{2} \cdot B_\varphi \left( \ell' + t \cdot \ell, \ell' + t \cdot \ell \right) = B_\varphi(\ell, \ell')\end{equation*}
on $V^*$; to get the last step, we expand $B_\varphi \left( \ell' + t \cdot \ell, \ell' + t \cdot \ell \right)$ by bilinearity into a quadratic polynomial in $t$ before differentiating it. We see that $\ell \circ \varphi = 0$ if and only if $\ell$ is in the nullspace of the symmetric bilinear form $B_\varphi$. This nullspace is exactly the kernel of the induced map $\varphi : V^* \to V$. Taking orthogonal complements, we conclude that $M(\varphi) = A(\varphi)^\perp = \im \varphi$, as required.
\end{example}

We are now ready to define our stratification $\mathbb{P}\left( \mathrm{Sym}^\mathbf{d} V \right) \setminus Z(\mathbf{f}^*) \supset Y_0 \supset \cdots \supset Y_k \supset Y_{k + 1} \supset \cdots$. In this endeavor, we directly generalize Tu \cite{Tu:1989aa}; that is, we simply replace $\rank(\varphi)$ with $m(\bm{\varphi})$.
\begin{definition}
For each $k \geq 0$, set $Y_k \coloneqq \set*{ [\bm{\varphi}] \in \mathbb{P}\left( \mathrm{Sym}^\mathbf{d} V \right) ; m(\bm{\varphi}) \leq r + 1 - k} \cap \left( \mathbb{P}\left( \mathrm{Sym}^\mathbf{d} V \right) \setminus Z(\mathbf{f}^*) \right)$.
\end{definition}

The following proposition also adapts an observation of Tu \cite{Tu:1989aa} to our setting.
\begin{proposition} \label{fiber_dimension}
For each $k \in \{0, \ldots , r \}$ and each element $[\bm{\varphi}] \in Y_k \setminus Y_{k + 1}$, $\dim h^{-1}\left( \{ [ \bm{\varphi} ] \} \right) = k \cdot (n - r)$. 
\end{proposition}
\begin{proof}
As Lemma \ref{universal} shows, the fiber $h^{-1}\left( \{ [ \bm{\varphi} ] \} \right) \subset \mathbb{P}\left( \text{Sym}^\mathbf{d} S \right)$ corresponds in a one-to-one way with the set of $r + 1$-dimensional subspaces $\Lambda \subset V$ for which $M(\bm{\varphi}) \subset \Lambda$. Our hypothesis $[\bm{\varphi}] \in Y_k \setminus Y_{k + 1}$ entails that $\dim M(\bm{\varphi}) = m(\bm{\varphi}) = r + 1 - k$. The set of such subspaces $\Lambda$ thus further corresponds in a one-to-one way with the set of $k$-dimensional subspaces of $V \mathbin{/} M(\bm{\varphi})$. This latter set is the Grassmannian $G(k, n - r + k)$, a variety whose dimension is $k \cdot (n - r)$.
\end{proof}

We note that if $k > r$, then $Y_k$ is necessarily empty, since $\dim M(\bm{\varphi}) = m(\bm{\varphi})$ is at least 1 whenever $\bm{\varphi} \neq 0$ (indeed, we must have $\bm{\varphi} \in \text{Sym}^\mathbf{d} M(\bm{\varphi})$).  On the other hand, for each $k < 0$, while the set $Y_k \subset \mathbb{P}\left( \mathrm{Sym}^\mathbf{d} V \right)$ makes sense, $h^{-1}\left( \{ [ \bm{\varphi} ] \} \right) = \varnothing$ holds for each $[\bm{\varphi}] \in Y_k \setminus Y_{k + 1}$ and each such $k$ (this follows from our proof of Proposition \ref{fiber_dimension}). In particular, $h : \mathbb{P}\left( \text{Sym}^\mathbf{d} S \right) \setminus Z(\bm{\mathfrak{f}}^*) \to \mathbb{P}\left( \mathrm{Sym}^\mathbf{d} V \right) \setminus Z(\mathbf{f}^*)$ is surjective onto $Y_0$.

\begin{proposition} \label{hard} 
Each $Y_k$ is closed and affine, of dimension at most $(r + 1 - k) \cdot (n - r + k) + {\mathbf{d} + r - k \choose r - k} - 1$.
\end{proposition}
\begin{proof}
For the purposes of this proof, we introduce a notational device. For each $k \in \{0, \ldots , r\}$, we write $Y'_k \coloneqq \set*{ [\bm{\varphi}] \in \mathbb{P}\left( \mathrm{Sym}^\mathbf{d} V \right) ; m(\bm{\varphi}) \leq r + 1 - k}$; in particular, we have $Y_k = Y'_k \cap \left( \mathbb{P}\left( \mathrm{Sym}^\mathbf{d} V \right) \setminus Z(\mathbf{f}^*) \right)$. 

We have the incidence correspondence:
\begin{figure}[H]
\centering
\begin{tikzpicture} 
  \node (top) at (2, 2) {$I_k = \set*{ \left( W, [\bm{\varphi}] \right) \in G \Big( r + 1 - k, n + 1\Big) \times \mathbb{P}\Big( \text{Sym}^{\mathbf{d}} V \Big) ; [\bm{\varphi}] \in \mathbb{P}\Big( \text{Sym}^\mathbf{d} W \Big)}$};
  \node (left) at (0, 0) {$G\Big(r + 1 - k, n + 1\Big)$};
  \draw[->] ($(top.south) + (-0.5, 0)$) -- (left.north);
  \node (right) at (5, 0) {$\mathbb{P}\Big( \text{Sym}^\mathbf{d} V \Big) \supset Y'_k$};
  \draw[->] ($(top.south) + (0.5, 0)$) -- node[right=\arrowlabelsep] {$\pi_k$} ($(right.north) + (-0.5, 0)$);
\end{tikzpicture}.
\end{figure}
This correspondence is in a sense dual to one which appears in \cite[\S~2]{Debarre:1998aa}. We first claim that the image $\pi_k(I_k) \subset \mathbb{P}\left( \text{Sym}^\mathbf{d} V \right)$ is exactly $Y'_k$. Indeed, $[\bm{\varphi}] \in Y'_k$ if and only if $m(\bm{\varphi}) = \dim M(\bm{\varphi})$ is \textit{at most} $r + 1 - k$; this latter condition holds if and only if there exists a subspace $W \subset V$ of dimension \textit{exactly} $r + 1 - k$ such that $\bm{\varphi} \in \text{Sym}^\mathbf{d} W$ (indeed, we can take as $W$ any $r + 1 - k$-dimensional subspace containing $M(\bm{\varphi})$).

The incidence correspondence above realizes $I_k$ as a projective bundle over $G(r + 1 - k, n + 1)$, whose fiber over $W \subset V$ is the projective space $\mathbb{P}\left( \text{Sym}^\mathbf{d} W \right)$. We conclude that $I_k$'s dimension is $(r + 1 - k) \cdot (n - r + k) + {\mathbf{d} + r - k \choose r - k} - 1$. 
Because $\pi_k$ is proper, $Y'_k \subset \mathbb{P}\left( \text{Sym}^\mathbf{d} V \right)$ is closed, of dimension at most that of $I_k$.

Since $\mathbb{P} \left( \mathrm{Sym}^\mathbf{d} V \right) \setminus Z(\mathbf{f}^*)$ is affine and $Y'_k \subset \mathbb{P} \left( \mathrm{Sym}^\mathbf{d} V \right)$ is closed, $Y_k = Y'_k \cap \left( \mathbb{P} \left( \mathrm{Sym}^\mathbf{d} V \right) \setminus Z(\mathbf{f}^*) \right)$ is itself closed and affine, of dimension at most that of $Y'_k$.
\end{proof}
\begin{remark} \label{hypothesis}
Though Proposition \ref{hard} isolates just what we need for Proposition \ref{final}, the full situation is rather nicer (unless $d_1 = \cdots = d_s = 1$, a case we presently exclude).
Indeed, it's possible to show that for each $k \in \{0, \ldots , r\}$, the sets $I_k$, $Y'_k$, and $Y_k$ are all nonempty, irreducible, and of dimension \textit{exactly} equal to that of $I_k$. That $I_k$ and its image $Y'_k$ are irreducible is essentially immediate. To show that $\dim I_k = \dim Y'_k$, it suffices to exhibit an explicit element $[\bm{\varphi}] \in Y'_k \setminus Y'_{k + 1}$. This feat can itself be done using apolarity. The existence of just such an element implies that the open set $Y'_k \setminus Y'_{k + 1} \subset Y'_k$ is nonempty, and so dense; moreover, the fiber $\pi_k^{-1}\left( \{[\bm{\varphi}]\}\right) = \left\{ (M(\bm{\varphi}), [\bm{\varphi}]) \right\}$ is 0-dimensional for each such element $[\bm{\varphi}] \in Y'_k \setminus Y'_{k + 1}$. By the theorem on fiber dimensions, the set consisting of those $[\bm{\varphi}] \in Y'_k$ for which $\dim \pi_k^{-1}\left( \{[\bm{\varphi}]\}\right) = \dim I_k - \dim Y'_k$ itself contains an open dense subset of $Y'_k$, which must therefore intersect $Y'_k \setminus Y'_{k + 1}$ nontrivially. These facts imply that $\dim I_k = \dim Y'_k$. Finally, $Y'_k$ isn't contained in any hyperplane, a fact which follows from the representation theory of the $GL(V)$-representation $\text{Sym}^\mathbf{d} V$. This final fact guarantees that $Y'_k$'s open affine hyperplane complement $Y_k$ is necessarily nonempty, again irreducible and of the same dimension as $Y'_k$.

\end{remark}

We're now ready to prove our main result. As Proposition \ref{fiber_dimension} shows that we must, we set $d(k) \coloneqq k \cdot (n - r)$. In this context, we obtain the following result, which---in light of Lemmas \ref{topological} and \ref{comparison}---suffices to complete our proof of Theorem \ref{main}. 
\begin{proposition} \label{final}
If $\delta\underscore \geq 0$, $R \coloneqq \max_{k \in \{0, \ldots , r\}} \left( \dim Y_k + 2 \cdot d(k) \right)$ satisfies $R < 2 \cdot \dim G(r + 1, n + 1) - \delta\underscore$.
\end{proposition}
\begin{proof}
Unwinding what the statement of the proposition means and using Propositions \ref{fiber_dimension} and \ref{hard}, we find that the claim
\begin{equation}\label{thingy}\max_{k \in \{0, \ldots , r\}} \left( (r + 1 - k) \cdot (n - r + k) + {\mathbf{d} + r - k \choose r - k} - 1 + 2 \cdot k \cdot (n - r) \right) < 2 \cdot \dim G(r + 1, n + 1) - \delta\underscore\end{equation}
suffices to prove the proposition. Indeed, Proposition \ref{hard} implies that the left-hand side of \eqref{thingy} above itself upper-bounds $R$ (in fact, this inequality is \textit{a posteriori} an equality, as we note in Remark \ref{hypothesis} above).

We note first that
\begin{equation}\label{sum}\Big( (r + 1 - k) \cdot (n - r + k) + 2 \cdot k \cdot (n - r) \Big) + \Big( (r + 1 - k) \cdot (n - r - k) \Big) = 2 \cdot \dim G(r + 1, n + 1).\end{equation}
Taking $2 \cdot \dim G(r + 1, n + 1)$ minus both sides of \eqref{thingy} and using \eqref{sum}, we obtain the following equivalent reformulation of \eqref{thingy}:
\begin{equation}\label{equivalent}\min_{k \in \{0, \ldots , r\}} \left( (r + 1 - k) \cdot (n - r - k) - {\mathbf{d} + r - k \choose r - k} + 1 \right) > \delta\underscore.\end{equation}
Precisely this claim is asserted in \cite[(2.9)]{Debarre:1998aa}. Indeed, \cite[(2.9)]{Debarre:1998aa} is equivalent to our \eqref{equivalent} under the reparameterization $h = r + 1 - k$. As those authors note, the values of \eqref{equivalent} at $k = 0$ and $k = r$ are $\delta + 1$ and $n - 2r - s + 1$, respectively; it thus suffices to show that, over the range $k \in \{0, \ldots , r\}$,
\begin{equation}\label{f}F(k) \coloneqq (r + 1 - k) \cdot (n - r - k) - {\mathbf{d} + r - k \choose r - k} + 1\end{equation}
attains its minimum value at one of its two endpoints.

It's harmless to assume that $d_i > 1$ for each $i \in \{1, \ldots , s\}$. As Debarre and Manivel argue, \eqref{f} is \textit{decreasing} on $k \in \{0, \ldots , r\}$ if $\mathbf{d} = (2)$ and \textit{concave} if $\mathbf{d} \neq (2)$. For self-containedness, we record a proof of this fact here.

To prove the first among these claims, we note that, in the case $\mathbf{d} = (2)$, for each $k \in \{1, \ldots , r\}$, $F(k - 1) - F(k) = n - k - r$ (this is a trivial algebraic simplification). This quantity itself is at least $n - 2r > n - 2r - s \geq \delta\underscore \geq 0$, a fact which shows that $F(k)$ is (strictly) decreasing on $\{0, \ldots , r\}$. Here (and nowhere else in this proof), we use our hypothesis $\delta\underscore \geq 0$.


To show that \eqref{f} is concave in case $\mathbf{d} \neq (2)$, it suffices to show that, for each $k \in \{1, \ldots , r - 1\}$, the discrete second difference
\begin{equation*}\Delta^2 F(k) = F(k + 1) - 2 \cdot F(k) + F(k - 1) \leq 0.\end{equation*}
Since $F(k)$'s first summand $(r + 1 - k) \cdot (n - r - k)$ is quadratic and monic in $k$, that summand's discrete second difference is identically 2. On the other hand, writing $G(k) \coloneqq {\mathbf{d} + r - k \choose r - k}$ for \eqref{f}'s subtrahend, for each $k \in \{1, \ldots , r - 1\}$, repeatedly using Pascal's identity ${N \choose j} = {N - 1 \choose j - 1} + {N - 1 \choose j}$, we obtain:
\begin{align*}\Delta^2 G(k) &= {\mathbf{d} + r - k - 1 \choose r - k - 1} - 2 \cdot {\mathbf{d} + r - k \choose r - k} + {\mathbf{d} + r - k + 1 \choose r - k + 1} \tag{by definition.} \\
&= {\mathbf{d} + r - k \choose r - k} - {\mathbf{d} + r - k - 1 \choose r - k} - 2 \cdot {\mathbf{d} + r - k \choose r - k} + {\mathbf{d} + r - k \choose r - k} + {\mathbf{d} + r - k \choose r - k + 1} \tag{Pascal.} \\
&= {\mathbf{d} + r - k \choose r - k + 1} - {\mathbf{d} + r - k - 1 \choose r - k} \tag{cancel terms.} \\
&= {\mathbf{d} + r - k - 1 \choose r - k + 1}. \tag{Pascal's identity one last time.}
\end{align*}
It suffices to demonstrate that this last term above satisfies ${\mathbf{d} + r - k - 1 \choose r - k + 1} \geq 2$. For each $k \in \{1, \ldots , r - 1\}$ and each $i \in \{1, \ldots , s\}$, ${d_i + r - k - 1 \choose r - k + 1} = {d_i + r - k - 1 \choose d_i - 2}$ is positive, and equals 1 precisely when $d_i = 2$. Our hypothesis $\mathbf{d} \neq (2)$ guarantees that \textit{either} the sum ${\mathbf{d} + r - k - 1 \choose r - k + 1} = \sum_{i = 1}^s {d_i + r - k - 1 \choose r - k + 1}$ has more than one term \textit{or} its sole term satisfies $d_i > 2$; in either case, we obtain  $\Delta^2 G(k) = {\mathbf{d} + r - k - 1 \choose r - k + 1} \geq 2$, which completes the proof.
\end{proof}
\begin{remark}
As Remark \ref{hypothesis} and the proof of Proposition \ref{final} jointly show, the inequality in that proposition's statement holds \textit{sharply}; that is, $R = 2 \cdot \dim G(r + 1, n + 1) - \delta\underscore - 1$ in fact holds in each case. This latter fact reflects the state of affairs whereby our Theorem \ref{main} is itself often sharp, as Debarre and Manivel note already in the rational case \cite[Rem.~3.6.~1)]{Debarre:1998aa}.
\end{remark}

\begin{remark}
We find it quite remarkable that the cohomology comparison claim \eqref{thingy} above---itself an indirect consequence of the Leray spectral sequence \cite[Lem.~3.6]{Tu:1989aa}---amounts exactly to \cite[(2.9)]{Debarre:1998aa}, an expression obtained using completely different means (and which further \textit{reappears} in the proof of \cite[Prop.~3.8]{Debarre:1998aa}). Undoubtedly, there is some deeper geometric story to tell here---say, a geometric interpretation of $F(k)$ for \textit{each} $k \in \{0, \ldots , r\}$---though we don't, admittedly, see it now.
\end{remark}

\begin{remark} \label{new}
Interestingly, the entire course of this paper continues to go through unchanged even in the case $r = 0$. In that case, the Grassmannian $G(r + 1, n + 1) = \mathbb{P}^n$ is simply $n$-dimensional projective space, and the Fano scheme $F_r(X) = X$. Our main Theorem \ref{main}, in that case, recovers the weak Lefschetz theorem for intersections of ample divisors, itself a special case of \textit{Sommese's theorem} (see e.g.\ Lazarsfeld \cite[Rem.~3.1.32]{Lazarsfeld:2004aa}). Our work thus yields a new proof of that result; in fact, it realizes that result as a special case of our more general theorem, which allows the projective linear subspaces at issue---i.e., those contained in $X \subset \mathbb{P}^n$---to be arbitrary-dimensional (i.e., and in particular $0$-dimensional).
\end{remark}

\paragraph{Acknowledgements.} I would like to sincerely thank Claire Voisin for first suggesting to me the problem that grew into this work, as well as for helping me in many further ways.



\printbibliography

@book{Lazarsfeld:2004aa,
	author = {Lazarsfeld, Robert},
	editor = {Remmert, R.},
	publisher = {Springer},
	series = {Ergebnisse der Mathematik und ihrer Grenzgebiete},
	title = {Positivity in Algebraic Geometry I},
	volume = {48},
	year = {2004}}

@book{Lazarsfeld:2004ab,
	author = {Lazarsfeld, Robert},
	editor = {Remmert, R.},
	publisher = {Springer},
	series = {Ergebnisse der Mathematik und ihrer Grenzgebiete},
	title = {Positivity in Algebraic Geometry II},
	volume = {49},
	year = {2004}}

@book{Eisenbud:1995aa,
	author = {Eisenbud, David},
	editor = {Ewing, J. H. and Gehring, F. W. and Halmos, P. R.},
	publisher = {Springer-Verlag},
	series = {Graduate Texts in Mathematics},
	title = {Commutative Algebra with a View Toward Algebraic Geometry},
	volume = {150},
	year = {1995}}

@article{Sommese:1978aa,
	author = {Sommese, Andrew John},
	doi = {10.1007/BF01405353},
	isbn = {1432-1807},
	journal = {Mathematische Annalen},
	number = {3},
	pages = {229--256},
	title = {Submanifolds of Abelian varieties to Rebecca},
	url = {https://doi.org/10.1007/BF01405353},
	volume = {233},
	year = {1978}}

@article{Tu:1990aa,
	author = {Loring W. Tu},
	issn = {0137-6934},
	journal = {Banach Center Publications},
	number = {2},
	pages = {235-248},
	title = {The connectedness of degeneracy loci},
	url = {http://eudml.org/doc/268173},
	volume = {26},
	year = {1990}}

@book{Iarrobino:1999aa,
	author = {Iarrobino, Anthony and Kanev, Vassil},
	editor = {Dold, A. and Takens, F. and Teissier, B.},
	publisher = {Springer},
	series = {Lecture Notes in Mathematics},
	title = {Power Sums, Gorenstein Algebras, and Determinantal Loci},
	volume = {1721},
	year = {1999}}

@article{Tu:1989aa,
	author = {Tu, Loring W.},
	journal = {Transactions of the American Mathematical Society},
	number = {1},
	pages = {381--392},
	title = {The Connectedness of Symmetric and Skew-Symmetric Degeneracy Loci: Even Ranks},
	volume = {313},
	year = {1989}}

@article{Debarre:1998aa,
	author = {Debarre, Olivier and Manivel, Laurent},
	journal = {Mathematische Annalen},
	number = {3},
	pages = {549--574},
	title = {Sur la vari{\'e}t{\'e} des espaces lin{\'e}aires contenus dans une intersection compl{\`e}te},
	volume = {312},
	year = {1998}}

\end{document}